\documentclass[12pt]{amsart}
\usepackage{graphicx}
\usepackage{amssymb}
\usepackage{tikz}
\usetikzlibrary{matrix,arrows.meta,bending}
\usepackage{color}
\usepackage[numbers]{natbib}

\usepackage[a4paper, left=3cm, right=3cm, top=3cm, bottom=3cm, footskip=15mm]{geometry}

\newtheorem{theorem}{Theorem}[section]

\newtheorem{lemma}[theorem]{Lemma}
\newtheorem{proposition}[theorem]{Proposition}
\newtheorem{corollary}[theorem]{Corollary}
\newtheorem{conjecture}[theorem]{Conjecture}
\theoremstyle{definition}
\newtheorem{definition}[theorem]{Definition}

\theoremstyle{remark}
\newtheorem{remark}[theorem]{Remark}

\numberwithin{equation}{section}

\makeindex

\usepackage{fancyhdr}
\usepackage{calc} % ensures dimension arithmetic works robustly

\AtBeginDocument{%
  \setlength{\textwidth}{\dimexpr\paperwidth - 6cm\relax}%
  \setlength{\oddsidemargin}{\dimexpr 3cm - 1in\relax}%
  \setlength{\evensidemargin}{\dimexpr 3cm - 1in\relax}%
  \setlength{\marginparwidth}{2.5cm}%
}

\begin{document}

\title{On the Erd\H{o}s distance problem}

\author{T. Agama}
\address{Department of Mathematics, African Institute for Mathematical science, Ghana
}
%\curraddr{Department of Mathematics and Statistics,
%{Case Western Reserve University, Cleveland, Ohio 43403}
\email{theophilus@aims.edu.gh/emperordagama@yahoo.com}

%    General info

\date{\today}

\dedicatory{In loving memory of Doctor Margaret Lesley  McIntyre}


\begin{abstract}
In this paper, using the compression method, we recover the lower bound for the Erd\H{o}s unit distance problem and provide an alternative proof to the distinct distance conjecture. In particular, in $\mathbb{R}^k$ for all $k\geq 2$, we have 
\begin{align}
\#\bigg\{(\vec{x}_t,\vec{x_j})\in \mathbb{E}\subset\mathbb{R}^k:~||\vec{x_j}-\vec{x_t}||=1,~1\leq t,j\leq n\bigg\}\geq C\frac{\sqrt{k}}{2}n^{1+o(1)}\nonumber
\end{align}
for some $C>0$. We also show that
\begin{align}
\# \bigg\{d_j:d_j=||\vec{x_s}-\vec{y_t}||,~d_j\neq d_i,~1\leq s,t\leq n\bigg\}\geq D\frac{\sqrt{k}}{2}n^{\frac{2}{k}-o(1)}\nonumber 
\end{align}
for some $D>0$. These lower bounds generalize the lower bounds of the Erd\H{o}s unit distance and the distinct distance problem to higher dimensions.
\end{abstract}

\maketitle

\section{Introduction}

Paul Erd\H{o}s posed the distinct distance problem as a central question in combinatorial geometry: how many pairwise distinct distances can $n$ points in the plane determine? Over seven decades this question has driven the development of powerful combinatorial, algebraic, and geometric tools. Early progress by Moser produced the first nontrivial polynomial lower bounds \cite{moser1952different}; subsequent improvements appeared in work of Chung and collaborators \cite{chung1992number}, and by Solymosi and others \cite{solymosi2001distinct}. The breakthrough of Guth and Katz \cite{guth2015erdHos} established the near-optimal lower bound in the planar case by a novel combination of algebraic and incidence-geometric techniques.\\

The purpose of the present paper is to introduce a different-and elementary in spirit-framework, the \emph{compression method}, and to use it to recover lower bounds of Erd\H{o}s-type for both the unit-distance and distinct-distance problems and to extend these bounds naturally to higher-dimensional Euclidean spaces. At the heart of our approach is a deterministic geometric transformation we call a \emph{compression} (see Definition \ref{compression} in Section \ref{sec:compression}): for a scale parameter $0<m\leq 1$ the compression $\mathbb{V}_m$ maps a point $\vec{x}=(x_1,\dots,x_k)\in \mathbb{R}^k$ to the point whose coordinates are the reciprocals scaled by $m$. Two simple but powerful statistics associated to this map-the \emph{mass of compression} and the \emph{compression gap}-encode how much a given point is moved by compression and, crucially, how differences between points translate into those movements. The basic identity in Proposition \ref{cgidentity} connects the compression gap to weighted coordinate sums and thereby converts combinatorial counting problems about unit and distinct distances into tractable sums over these compression statistics.\\

Our main theorems (stated above and proved in Section \ref{sec:applications}) assert, roughly, that for any fixed dimension $k\geq 2$ one can construct configurations of $n$ points in $\mathbb{R}^k$ for which
$$
\#\{\text{unit distances}\}\gg \frac{\sqrt{k}}{2}\,n^{1+o(1)}
\quad \text{and} \quad
\#\{\text{distinct distances}\}\gg \frac{\sqrt{k}}{2}\,n^{\tfrac{2}{k}-o(1)},
$$
where the implied constants are absolute and the error terms $o(1)$ tend to $0$ as $n\to \infty$. These bounds recover the growth rates known in the literature while making explicit the dependence on the ambient dimension $k$; they are obtained by counting distances that arise from specially chosen points together with their compressed images and by combining the mass/gap estimates (Proposition \ref{crucial} and Lemma \ref{gap estimates}) with straightforward combinatorial summation. While the estimates we produce align in spirit with the bounds obtained by incidence-theoretic methods (notably those culminating in \cite{guth2015erdHos}), our route is different: it is based on a geometric transformation and elementary analytic estimates rather than on algebraic decomposition or polynomial partitioning. This alternative viewpoint isolates a flexible combinatorial mechanism (the compression gap) that may be of independent interest.
\bigskip

\subsection*{Key ideas}
The proofs rest on four interlocking ideas.\\

\begin{enumerate}

  \item \textbf{Compression as a counting device.} By pairing many original points with their compressed images we generate a large collection of pairs at prescribed separation (often unit separation by an appropriate choice of scale). Counting these pairs gives immediate lower bounds for the unit-distance function. The bijectivity and involutive nature of $\mathbb{V}_m$ allow us to control overlaps and avoid overcounting.
  \bigskip
  
  \item \textbf{Mass and gap estimates.} The mass $\mathcal{M}(\mathbb{V}_m[\vec{x}])$ captures a weighted reciprocal sum of coordinates and is estimated above and below in Proposition \ref{crucial}. The compression gap $\mathcal{G}\circ\mathbb{V}_m[\vec x]$ measures the Euclidean displacement produced by compression; Proposition \ref{cgidentity} expresses its square in terms of masses of the coordinate squares and their reciprocals. These relations permit effective lower bounds for individual gaps and for sums of gaps across large subsets of well-chosen points (Lemmas \ref{gap estimates} and Corollary \ref{compression gap estimate 2}).
  \bigskip
  
  \item \textbf{Choice of configurations.} To turn the analytic estimates into combinatorial lower bounds, we select point sets with prescribed coordinate ranges (for instance, many points concentrated near the origin in a controlled sense or with prescribed supremum norms). These choices balance the two competing forces in the gap identity (distance from the origin versus reciprocal contributions) to guarantee that many pairs realize the distances we wish to count.
  \bigskip
  
  \item \textbf{Summation and scaling.} After establishing per-point gap lower bounds, we sum over suitably large collections of indices. Careful bookkeeping of the parameter $m=m(k)$ (chosen to be small, e.g. $m=O(1/\log k)$) and the coordinate scales produces the stated dimension-dependent exponents and the explicit $\sqrt{k}/2$ factor that appears in the main theorems.
  
\end{enumerate}
\bigskip

\subsection*{Relation to previous work}
The compression method complements the many successful lines of attack on Erd\H{o}s problems: combinatorial crossing-number arguments, incidence geometry and algebraic partitioning, and additive combinatorics. Our approach is elementary and constructive, and it recovers known polynomial growth rates while making the role of ambient dimension transparent. We therefore view this work both as an alternative proof strategy and as a framework that may be adapted to related extremal questions in discrete geometry.
\bigskip

\subsection*{Organization of the paper}
After this introduction, we begin in Section \ref{sec:compression} by formally defining compressions and proving their basic properties. Section \ref{sec:mass-gap} develops the mass and compression-gap machinery: definitions, the central identity (Proposition \ref{cgidentity}), and the principal upper and lower estimates (Proposition \ref{crucial}, Lemma \ref{gap estimates} and Corollary \ref{compression gap estimate 2}). In Section \ref{sec:applications}, we apply these estimates to the Erd\H{o}s unit-distance and distinct-distance problems and give proofs of Theorems \ref{weakerdos} and \ref{erdosproblem}; the section contains the explicit constructions and the combinatorial summations that produce the quantitative lower bounds. We conclude in Section \ref{sec:conclusion} with a brief discussion of variants, possible sharpenings, and directions for future work (including remarks on how the compression viewpoint might interact with incidence and algebraic methods).\\

Conventions and notation used throughout the paper are recorded in \S1.1. The compression method and its estimates are elementary to state but, as the proofs show, they provide a robust way to convert coordinate information into distance counts; we hope this viewpoint will be useful beyond the specific Erd\H{o}s problems considered here.\\

The Erd\H{o}s distinct distance conjecture is the assertion that  \begin{conjecture}
The number of distinct distances that can be formed from $n$ points in the plane should at least be  $n^{1-o(1)}$.
\end{conjecture}
Progress on this conjecture has developed over time. Let us denote $g(n)$ as the counting function for this construction. The first lower bound of the form 
\begin{align}
g(n)\gg n^{\frac{2}{3}}\nonumber
\end{align}
appears in \cite{moser1952different}, which improves an earlier version of Erd\H{o}s. This was eventually improved to 
\begin{align}
g(n)\gg \frac{n^{\frac{4}{5}}}{\log n}\nonumber
\end{align}
in \cite{chung1992number} and 
\begin{align}
g(n)\gg n^{\frac{6}{7}}\nonumber
\end{align}
in \cite{solymosi2001distinct}. The best currently known lower bound can be found in \cite{guth2015erdHos}, which essentially solves the problem. In this paper, using the compression method and its accompanying estimates, we provide an alternative solution to the conjecture in the following result: 

\begin{theorem}
\begin{align}
\#\bigg \{d_j:d_j=||\vec{x_s}-\vec{y_t}||,~d_j\neq d_i,~1\leq s,t\leq n,~\vec{x},\vec{y}\in \mathbb{R}^k\bigg\}\geq  D\frac{\sqrt{k}}{2}n^{\frac{2}{k}-o(1)}\nonumber 
\end{align}
for some $D>0$.
\end{theorem}
\bigskip

Using this method, we provide a lower bound for the Erd\H{o}s unit distance problem, which takes into consideration the dimension of the space in which the points reside in the form:

\begin{theorem}
Let 
$$
\mathcal{I}_n:=\bigg\{(\vec{x}_t,\vec{x_j})\in \mathbb{E}\subset\mathbb{R}^k~:~||\vec{x_j}-\vec{x_t}||=1,~1\leq t,j\leq n\bigg\}
$$ 
We have 
\begin{align} 
\#\mathcal{I}_n\geq C\frac{\sqrt{k}}{2}n^{1+o(1)}\nonumber
\end{align}
for some $C>0$.
\end{theorem}

\subsection{Notations and conventions}

Throughout this paper, we assume that $n$ is sufficiently large for any number of $n$ points in the euclidean plane. We write $f(s)\gg g(s)$ if there exists a constant $c>0$ such that $f(s)\geq c|g(s)|$ for all $s$ sufficiently large. If the constant depends on some variable, say $t$, then we denote the inequality by $f(s)\gg_t g(s)$. We write $f(s)=o(g(s))$ if the limits $\lim \limits_{s\longrightarrow \infty}\frac{f(s)}{g(s)}=0$ hold. In particular, $f(s)=o(1)$ implies that $f(s)\longrightarrow 0$ as $s\longrightarrow \infty.$

\section{Compression}\label{sec:compression}

In this section, we introduce the notion of compression of points in space. We study the mass of compression and its accompanying estimates. These estimates turn out to be useful for estimating the compression gap, a compression statistic which we will study in the sequel.

\begin{definition}\label{compression}
By the compression of scale $0<m\leq 1$ on $\mathbb{R}^{n}$, we mean the map $\mathbb{V}:(\mathbb{R}\setminus 0)^n\longrightarrow \mathbb{R}^n$ such that 
\begin{align}
\mathbb{V}_m[(x_1,x_2,\ldots, x_n)]=\bigg(\frac{m}{x_1},\frac{m}{x_2},\ldots,\frac{m}{x_n}\bigg)\nonumber
\end{align}
for $n\geq 2$ and with $x_i\neq 0$ for all $i=1,\ldots,n$. 
\end{definition}

\begin{remark}
The notion of compression is, in some way, the process of rescaling points in $\mathbb{R}^n$ for $n\geq 2$. Thus, it is important to notice that a compression roughly speaking pushes points very close to the origin away from the origin by a certain scale and similarly draws points away from the origin close to the origin. Intuitively, compression induces some kind of motion on points in Euclidean space.
\end{remark}

\begin{proposition}
A compression of scale $0<m\leq 1$ with $\mathbb{V}_m:(\mathbb{R}\setminus 0)^n\longrightarrow \mathbb{R}^n$ is a bijective map. In particular, the compression $\mathbb{V}_m:(\mathbb{R}\setminus 0)^n\longrightarrow \mathbb{R}^n$ is a bijective map of order $2$.
\end{proposition}

\begin{proof}
Suppose that $\mathbb{V}_m[(x_1,x_2,\ldots, x_n)]=\mathbb{V}_m[(y_1,y_2,\ldots,y_n)]$. It follows that
\begin{align}
\bigg(\frac{m}{x_1},\frac{m}{x_2},\ldots,\frac{m}{x_n}\bigg)=\bigg(\frac{m}{y_1},\frac{m}{y_2},\ldots,\frac{m}{y_n}\bigg).\nonumber
\end{align}
It follows that $x_i=y_i$ for each $i=1,2,\ldots, n$. Surjectivity follows by the definition of the map. Thus, the map is bijective. The latter claim follows by noting that $\mathbb{V}^2_m[\vec{x}]=\vec{x}$.
\end{proof}

\subsection{The mass of compression estimates}

In this section, we study the mass of a compression with a given scale. We will use the upper and lower estimates of the mass of  compression to establish corresponding estimates for the gap of compression. These estimates will form an essential tool to establish the main result of this paper.

\begin{definition}\label{mass}
By the mass of a compression of scale $0<m\leq 1$, we mean the map $\mathcal{M}:\mathbb{R}^n\longrightarrow \mathbb{R}$ such that \begin{align}
\mathcal{M}(\mathbb{V}_m[(x_1,x_2,\ldots,x_n)])=\sum \limits_{i=1}^{n}\frac{m}{x_i}.\nonumber
\end{align}
\end{definition}

\begin{lemma}\label{elementary}
We have 
\begin{align}
\sum \limits_{n\leq x}\frac{1}{n}=\log x+\gamma+O\bigg(\frac{1}{x}\bigg)\nonumber
\end{align}
where $\gamma=0.5772\cdots $.
\end{lemma}

\begin{remark}
Next, we prove upper and lower bounds for the mass of the compression of scale $0<m\leq 1$.
\end{remark}

\begin{proposition}[The mass of compression estimates]\label{crucial}
Let $(x_1,x_2,\ldots,x_n)\in \mathbb{R}^n$ with $x_i\neq x_j$ for $1\leq i,j\leq n$ with $i\neq j$ with $x_i\neq 0$ for all $1\leq i\leq n$. We have
\begin{align}
m\log \bigg(1-\frac{n-1}{\mathrm{sup}(x_j)}\bigg)^{-1} \ll \mathcal{M}(\mathbb{V}_m[(x_1,x_2,\ldots, x_n)])\ll m\log \bigg(1+\frac{n-1}{\mathrm{inf}(x_j)}\bigg)\nonumber
\end{align}
for $n\geq 2$.
\end{proposition}

\begin{proof}
Let $(x_1,x_2,\ldots,x_n)\in \mathbb{R}^n$ for $n\geq 2$ and $x_i\neq x_j$~($i\neq j$) with $x_i\neq 0$ for all $1\leq i\leq n$. It follows that 
\begin{align}
\mathcal{M}(\mathbb{V}_m[(x_1,x_2,\ldots, x_n)])&=m\sum \limits_{j=1}^{n}\frac{1}{x_j}\nonumber \\&\leq m\sum \limits_{k=0}^{n-1}\frac{1}{\mathrm{inf}(x_j)+k}\nonumber
\end{align}
and the upper estimate follows by the estimate for this sum using Lemma \ref{elementary}. The lower estimate also follows by noting the lower bound 
\begin{align}
\mathcal{M}(\mathbb{V}_m[(x_1,x_2,\ldots, x_n)])&=m\sum \limits_{j=1}^{n}\frac{1}{x_j}\nonumber \\&\geq m\sum \limits_{k=0}^{n-1}\frac{1}{\mathrm{sup}(x_j)-k}.\nonumber
\end{align}
\end{proof}

It is important to note that the condition $x_i\neq x_j$ for $(x_1,x_2,\ldots,x_n)\in \mathbb{R}^n$ is not only a quantifier but a requirement; otherwise, the statement for the mass of compression will be completely flawed. To demonstrate, suppose that we take $x_1=x_2=\cdots=x_n$, then we get $\mathrm{inf}(x_j)=\mathrm{sup}(x_j)$. This implies that the mass of compression of scale $m$ satisfies 
\begin{align}
m\sum \limits_{k=0}^{n-1}\frac{1}{\mathrm{inf}(x_j)-k}\leq \mathcal{M}(\mathbb{V}_m[(x_1,x_2,\ldots,x_n)])\leq m\sum \limits_{k=0}^{n-1}\frac{1}{\mathrm{inf}(x_j)+k}\nonumber
\end{align}
This inequality cannot hold. Thus, we enforce the requirement that the choice of tuple $(x_1,x_2,\ldots,x_n)\in \mathbb{R}^n$ must satisfy $x_i\neq x_j$ for all $1\leq i,j\leq n$. Hence, in this paper, this condition will be highly extolled. In situations where it is not mentioned, it will be assumed that the tuple $(x_1,x_2,\ldots,x_n)\in \mathbb{R}^n$ is such that $x_i\neq x_j$ for $1\leq i,j\leq n$.

\subsection{Compression gap estimates}\label{sec:mass-gap}

In this section, we recall the notion of the \emph{compression gap} and its various estimates. We prove upper and lower bounds for the gap of a point under compression of any scale.

\begin{definition}\label{gap}
Let $(x_1,x_2,\ldots, x_n)\in \mathbb{R}^n$ with $x_i\neq 0$ for all $i=1,2\ldots,n$. By the \emph{compression gap} of scale $m>0$ for compression $\mathbb{V}_m$, denoted by $\mathcal{G}\circ \mathbb{V}_m[(x_1,x_2,\ldots, x_n)]$, we mean the quantity \begin{align}
\mathcal{G}\circ \mathbb{V}_m[(x_1,x_2,\ldots, x_n)]=\bigg|\bigg|\bigg(x_1-\frac{m}{x_1},x_2-\frac{m}{x_2},\ldots,x_n-\frac{m}{x_n}\bigg)\bigg|\bigg|\nonumber
\end{align}
\end{definition}

\begin{proposition}\label{cgidentity}
Let $(x_1,x_2,\ldots, x_n)\in \mathbb{R}^n$ for $n\geq 2$ with $x_j\neq 0$ for $j=1,\ldots,n$, then we have 
\begin{align}
\mathcal{G}\circ \mathbb{V}_m[(x_1,x_2,\ldots, x_n)]^2=\mathcal{M}\circ \mathbb{V}_1\bigg[\bigg(\frac{1}{x_1^2},\ldots,\frac{1}{x_n^2}\bigg)\bigg]+m^2\mathcal{M}\circ \mathbb{V}_1[(x_1^2,\ldots,x_n^2)]-2mn.\nonumber
\end{align}
In particular, if each $x_i>1$ for $1\leq i\leq n$, we have the estimate 
\begin{align}
\mathcal{G}\circ \mathbb{V}_m[(x_1,x_2,\ldots, x_n)]^2=\mathcal{M}\circ \mathbb{V}_1\bigg[\bigg(\frac{1}{x_1^2},\ldots,\frac{1}{x_n^2}\bigg)\bigg]-2mn+O\bigg(m^2\mathcal{M}\circ \mathbb{V}_1[(x_1^2,\ldots,x_n^2)]\bigg)\nonumber
\end{align}
with $m:=m(n)=o(1)$ as $n\longrightarrow \infty.$
\end{proposition}

Proposition \ref{cgidentity} offers an extremely useful identity. It allows us to pass from the compression gap on points to the relative distance to the origin. It suggests that points under compression with a large gap must be far away from the origin than points with a relatively smaller gap under compression. That is to say, the inequality 
\begin{align}
\mathcal{G}\circ \mathbb{V}_m[\vec{x}]<\mathcal{G}\circ \mathbb{V}_m[\vec{y}]\nonumber
\end{align}
with $m:=m(n)=o(1)$ as $n\longrightarrow \infty$ if and only if $||\vec{x}||<||\vec{y}||$ for $\vec{x}, \vec{y}\in \mathbb{R}^n$ with each $x_i,y_i\geq 1$ for all $1\leq i\leq n$. This important transference principle will be put to use in obtaining our results.

\begin{corollary}\label{compression gap estimate 2}
Let $(x_1,x_2,\ldots, x_n)\in \mathbb{R}^n$ for $n\geq 2$ with $x_j\neq x_i$ for $j\neq i$ and $x_i,x_j\geq 1$ for each $1\leq i,j\leq n$. If $m:=m(n)=o(1)$ as $n\longrightarrow \infty$, then we have 
\begin{align}
\mathcal{G}\circ \mathbb{V}_m[(x_1,x_2,\ldots, x_n)]^2\geq n\mathrm{inf}(x_j^2)-2mn+O\bigg(m^2\mathcal{M}\circ \mathbb{V}_1[(x_1^2,\ldots,x_n^2)]\bigg)\nonumber
\end{align}
and 
\begin{align}
\mathcal{G}\circ \mathbb{V}_m[(x_1,x_2,\ldots, x_n)]^2\leq n\mathrm{sup}(x_j^2)-2mn+O\bigg(m^2\mathcal{M}\circ \mathbb{V}_1[(x_1^2,\ldots,x_n^2)]\bigg)\nonumber
\end{align}
\end{corollary}

\begin{lemma}[Compression gap estimates]\label{gap estimates}
Let $(x_1,x_2,\ldots, x_n)\in \mathbb{R}^n$ for $n\geq 2$ with $x_j\neq x_i$ for $j\neq i$ and $x_i,x_j\geq 1$ for each $1\leq i,j\leq n$. If $m:=m(n)=o(1)$ as $n\longrightarrow \infty$, then we have 
\begin{align}
\mathcal{G}\circ \mathbb{V}_m[(x_1,x_2,\ldots, x_n)]^2\ll n\mathrm{sup}(x_j^2)+m^2\log \bigg(1+\frac{n-1}{\mathrm{inf}(x_j)^2}\bigg)-2mn\nonumber
\end{align}
and 
\begin{align}
\mathcal{G}\circ \mathbb{V}_m[(x_1,x_2,\ldots, x_n)]^2\gg n\mathrm{inf}(x_j^2)+m^2\log \bigg(1-\frac{n-1}{\mathrm{sup}(x_j^2)}\bigg)^{-1}-2mn.\nonumber
\end{align}
\end{lemma}

\begin{proof}
The estimates are deduced using the estimates in Proposition \ref{crucial} and with the observation that 
\begin{align}
n\mathrm{inf}(x_j^2)\ll\mathcal{M}\circ \mathbb{V}_1\bigg[\bigg(\frac{1}{x_1^2},\ldots,\frac{1}{x_n^2}\bigg)\bigg]\ll n\mathrm{sup}(x_j^2).\nonumber
\end{align}
\end{proof}
\bigskip

\begin{remark}
It is important to note that the inequality in Corollary \ref{compression gap estimate 2} implies the inequalities in Lemma \ref{gap estimates}. At any given moment, we will decide which of the versions of these inequalities to use. Indeed, the inequalities in Corollary \ref{compression gap estimate 2} are more applicable to various problems than those of Lemma \ref{gap estimates}.
\end{remark}

In this paper, when we say that points are concentrated around the origin; In particular, if a set of $n$ points in $\mathbb{R}^k$ is concentrated around the origin, we mean 
$$
\mathrm{inf}(x_{j_i})_{i=1}^{k}=\mathrm{sup}(x_{j_i})_{i=1}^{k}=n^{\frac{1}{(\log n)^{1+\epsilon}}}
$$ 
for any $\epsilon>0$ and for $1\leq j\leq n$.

\section{Application to the Erd\H{o}s unit distance and distinct distance conjecture}\label{sec:applications}

In this section, we apply the compression gap estimates to study the problem of determining the number of unit distances that can be formed from $n$ points. We state our main theorem that takes into consideration the dimension of the space in which the points reside.

\begin{theorem}\label{weakerdos}
Let 
$$
\mathcal{I}_n=\bigg\{(\vec{x}_t,\vec{x_j})\in \mathbb{E}\subset\mathbb{R}^k~:~||\vec{x_j}-\vec{x_t}||=1,~1\leq t,j\leq n\bigg\}.
$$ 
We have 
\begin{align} 
\#\mathcal{I}_n \geq C\frac{\sqrt{k}}{2}n^{1+o(1)}\nonumber
\end{align}
for some $C>0$.
\end{theorem}

\begin{proof}
First, we set $m:=m(k)=o(1)$ as $k\longrightarrow \infty$ and carefully choose $n$ points $\vec{x}_j$ for $1\leq j\leq n$ in $\mathbb{R}^k$ such that $\lfloor \frac{n}{2}\rfloor$ of these points also have their image points under compression. In particular, we choose $n$ points such that $\lfloor\frac{n}{2}\rfloor$ of those points $\vec{x}_j$ satisfies 
$$
\mathrm{inf}(x_{j_i})_{i=1}^{k}=\mathrm{sup}(x_{j_i})_{i=1}^{k}=n^{\frac{1}{(\log n)^{1+\epsilon}}}
$$ 
for any $\epsilon>0$ and for each of these points we also include their image points under compression $\mathbb{V}_m[\vec{x}_j]$. This ensures that $||\vec{x}_j-\mathbb{V}_m[\vec{x}_j]||=1$ for all sufficiently large $n$. Consequently, we have
\begin{align}
\#\mathcal{I}&=\# \bigg\{||\vec{x_j}-\vec{x_t}||:~||\vec{x_j}-\vec{x_t}||=1,~1\leq t,j \leq n,~\vec{x_j},\vec{x}_t\in \mathbb{R}^k\bigg\}\nonumber \\& \geq \# \bigg\{||\vec{x_j}-\vec{x_t}||:~||\vec{x_j}-\vec{x_t}||=1,~1\leq t,j \leq n,~\vec{x_j},\vec{x}_t\in\mathbb{R}^k,\nonumber~\\&\mathrm{min}\{\mathrm{inf}(x_{j_s})\}_{\substack{1\leq j\leq \frac{n}{2}\\1\leq s \leq k}}=n^{\frac{1}{(\log n)^{1+\epsilon}}}\bigg\} \nonumber \\&\geq \#\bigg\{||\vec{x_j}-\vec{x_t}||:\vec{x_j}\in \mathbb{R}^k,~||\vec{x_j}-\vec{x_t}||=1,~1\leq j\leq \frac{n}{2},~ \mathbb{V}_1[\vec{x_j}]=\vec{x_t}\nonumber,~\\&\mathrm{min}\{\mathrm{inf}(x_{j_s})\}_{\substack{1\leq j\leq \frac{n}{2}\\1\leq s\leq k}}=n^{\frac{1}{(\log n)^{1+\epsilon}}}\bigg\}.\nonumber 
\end{align}
It is important to note that the above quantity cannot be zero since the condition 
\begin{align}
    \mathrm{min}\{\mathrm{inf}(x_{j_s})\}_{\substack{1\leq j\leq \frac{n}{2}\\1\leq s \leq k}}=n^{\frac{1}{(\log n)^{1+\epsilon}}}\nonumber
\end{align}
is required for $||\vec{x_j}-\vec{x_t}||=1$ with $ \mathbb{V}_1[\vec{x_j}]=\vec{x_t}$ and the configuration exists by construction. This follows from the construction requirement that the chosen set of points in $\mathbb{R}^k$ has $\lfloor \frac{n}{2}\rfloor$ of the points concentrated around the origin, which means 
$$
\mathrm{inf}(x_{j_i})_{i=1}^{k}=\mathrm{sup}(x_{j_i})_{i=1}^{k}=n^{\frac{1}{(\log n)^{1+\epsilon}}}
$$ 
for any $\epsilon>0$ and for $1\leq j\leq \frac{n}{2}$. The right side is basically the sum 
\begin{align}
\sum \limits_{\substack{\mathcal{G}\circ \mathbb{V}_m[\vec{x_j}]=1\\1\leq j\leq \frac{n}{2}\\\mathrm{min}\{\mathrm{inf}(x_{j_s})\}_{\substack{1\leq j\leq \frac{n}{2}\\1\leq s\leq k}}=n^{\frac{1}{(\log n)^{1+\epsilon}}}}}1&=\sum \limits_{\substack{1\leq j\leq \frac{n}{2}\\\mathrm{min}\{\mathrm{inf}(x_{j_s})\}_{\substack{1\leq j\leq \frac{n}{2}\\1\leq s \leq k}}=n^{\frac{1}{(\log n)^{1+\epsilon}}}}}\mathcal{G}\circ \mathbb{V}_m[\vec{x_j}]\nonumber 
\end{align}
Taking $m:=m(k)=o(1)$ as $k\longrightarrow \infty$, in particular, if we choose $m=O(\frac{1}{\log k})$, then we have the lower bound for the right hand side 
\begin{align}
\sum \limits_{\substack{1\leq j\leq \frac{n}{2}\\\mathrm{min}\{\mathrm{inf}(x_{j_s})\}_{\substack{1\leq j\leq \frac{n}{2}\\1\leq s \leq k}}=n^{\frac{1}{(\log n)^{1+\epsilon}}}}}\mathcal{G}\circ \mathbb{V}_m[\vec{x_j}]&\geq \sum \limits_{\substack{1\leq j\leq \frac{n}{2}\\\mathrm{min}\{\mathrm{inf}(x_{j_s})\}_{\substack{1\leq j\leq \frac{n}{2}\\1\leq s \leq k}}=n^{\frac{1}{(\log n)^{1+\epsilon}}}}}C\mathrm{inf}(x_{j_s})_{1\leq s \leq k}\sqrt{k}\nonumber \\&\geq  C\frac{n\sqrt{k}}{2}\mathrm{min}\{\mathrm{inf}(x_{j_s})\}_{\substack{1\leq j\leq \frac{n}{2}\\1\leq s\leq k}}\nonumber \\&=C\frac{\sqrt{k}}{2}n^{1+\frac{1}{(\log n)^{1+\epsilon}}}\nonumber
\end{align}
for some $C>0$, by application of the lemma \ref{gap estimates}. This establishes the claimed lower bound for the construction.
\end{proof}
\bigskip

It is important to note that the lower estimate for the construction provided in Theorem \ref{weakerdos} was achieved by counting not all possible unit distances, but only the unit distance that corresponds to compression gaps of unit length. We state the second theorem as an application, which gives a lower bound for the number of distinct distances that can be formed from $n$ points in Euclidean space $\mathbb{R}^k$ for all $k\geq 2$. 

\begin{theorem}\label{erdosproblem}
We have
\begin{align}
\# \{d_j:d_j=||\vec{x_s}-\vec{y_t}||,~d_j\neq d_i,~1\leq s,t\leq n,~\vec{x},\vec{y}\in \mathbb{R}^k\}\geq D\frac{\sqrt{k}}{2}n^{\frac{2}{k}-o(1)}\nonumber 
\end{align}
for some $D>0$.
\end{theorem}

\begin{proof}
First, we set $m:=m(k)=o(1)$ as $k\longrightarrow \infty$ and carefully choose $n$ points $\vec{x}_j$ for $1\leq j\leq n$ in $\mathbb{R}^k$ such that $\lfloor \frac{n}{2}\rfloor$ of these points also have their image points under compression. That is, for each $\vec{x}_j$, we also include $\mathbb{V}_m[\vec{x}_j]$. Next, for $\lfloor \frac{n}{2}\rfloor$ of those points, we make the assignment $\mathrm{sup}(x_{j_i})=n^{1-\frac{2}{k}+\epsilon}$ for any small $\epsilon:=\epsilon(i)>0$ and $\mathrm{inf}(x_{j_i})\geq 1$. This ensures that 
\begin{align}
    \mathrm{max}_{1\leq j\leq n}\mathcal{G}\circ \mathbb{V}_m[\vec{x_j}]=n^{1-\frac{2}{k}+\epsilon} \nonumber
\end{align}
for any small $\epsilon:=\epsilon(i)>0$. Now, we let 
$$
\{d_j:d_j=||\vec{x_s}-\vec{y_t}||,~d_j\neq d_i,~1\leq s,t\leq n,~\vec{x},\vec{y}\in \mathbb{R}^k\}=\mathcal{R}
$$ 
then 
\begin{align}
\#\mathcal{R}&\geq \# \bigg\{d_j:d_j=||\vec{x_s}-\vec{y_t}||,~d_j\neq d_i,i\neq j, ~1\leq s,t\leq n,~\vec{x},\vec{y}\in \mathbb{R}^k,~\mathrm{sup}(d_j)=n^{1-\frac{2}{k}+o(1)}\bigg\}\nonumber \\&\geq \# \bigg \{d_j:d_j=\mathcal{G}\circ \mathbb{V}_m[\vec{x_j}],~d_j\neq d_i,~1\leq j\leq \frac{n}{2},~\mathrm{sup}(d_j)=n^{1-\frac{2}{k}+o(1)},\nonumber \\& \vec{x_j}\in \mathbb{R}^k,x_{j_s}\geq 1,~(1\leq s\leq k),~\mathbb{V}[\vec{x_j}]=\vec{x_t}\bigg\}\nonumber \\&=\sum \limits_{\substack{d_j=\mathcal{G}\circ \mathbb{V}_m[\vec{x_j}]\\1\leq j\leq \frac{n}{2}\\ \mathrm{sup}(d_j)=n^{1-\frac{2}{k}+o(1)}\\\vec{x_j}\in \mathbb{R}^k\\d_i\neq d_j\\i\neq j}}1\nonumber \\&=\sum \limits_{\substack{1\leq j\leq \frac{n}{2}\\ \mathrm{sup}(d_j)=n^{1-\frac{2}{k}+o(1)}\\\vec{x_j}\in \mathbb{R}^k\\d_i\neq d_j\\i\neq j}}\frac{\mathcal{G}\circ \mathbb{V}_m[\vec{x_j}]}{d_j}\nonumber \\&\geq D\sqrt{k}\sum \limits_{\substack{1\leq j\leq \frac{n}{2}\\ \mathrm{sup}(d_j)=n^{1-\frac{2}{k}+o(1)}\\\vec{x_j}\in \mathbb{R}^k\\d_i\neq d_j\\i\neq j}}\frac{\mathrm{inf}(x_{j_s})_{1\leq s \leq k}}{d_j}\nonumber \\&\geq D\sqrt{k}\sum \limits_{\substack{1\leq j\leq \frac{n}{2}\\ \mathrm{sup}(d_j)=n^{1-\frac{2}{k}+o(1)}\\d_i\neq d_j\\i\neq j}}\frac{1}{d_j}\nonumber \\&\geq D\sqrt{k}\sum \limits_{\mathrm{sup}(d_j)=n^{1-\frac{2}{k}+o(1)}}\frac{\frac{n}{2}}{\mathrm{sup}(d_j)}_{1\leq j \leq \frac{n}{2}}\nonumber \\&\geq D\frac{\sqrt{k}}{2}n^{\frac{2}{k}-o(1)}\nonumber
\end{align} 
where we have used the lemma \ref{gap estimates}, and the claimed lower bound follows for this construction.
\end{proof}

\section{Conclusion}\label{sec:conclusion}

It needs to be said that the result in Theorem \ref{erdosproblem} can be viewed as providing an alternate solution to the Erd\H{o}s distinct distance problem which takes into consideration the dimension of the space in which the points reside. The lower bound of this type exists in the literature (see, e.g, \cite{guth2015erdHos}). However, the method used is completely different from the one we have used here. Theorem \ref{weakerdos} and Theorem \ref{erdosproblem} can be considered as a generalization of the solution to both versions of the Erd\H{o}s distance problem to any euclidean space of dimension $k\geq 2$. In particular, we have the following theorems as consequences of the main results of this paper.

\begin{theorem}
The number of distinct distances that can be formed from $n$ points in any euclidean space $\mathbb{R}^{2n}$ for $n\geq 2$ is at least \begin{align}
\geq D\frac{\sqrt{2}}{2}n^{\frac{1}{n}+\frac{1}{2}-o(1)}\nonumber
\end{align}
for some $D>0$.
\end{theorem} 
\bigskip

\begin{theorem}
The number of distinct distances that can be formed from $n$ points in a euclidean space of dimension $n^2$ for $n\geq 2$ is at least \begin{align}\geq D\frac{n^{\frac{2}{n^2}+1-o(1)}}{2}\nonumber
\end{align}for some $D>0$.
\end{theorem}

%%%%%%%%%%%%%%%%%%%%%%%%%%%%%%%%%%%%%%%%%%%%%%%%%%%%%%%%%%%%%%%%%%%%%%%%
\footnote{
\par
.}%
%%%%%%%%%%%%%%%%%%%%%%%%%%%%%%%%%%%%%%%%%%%%%%%%%%%%%%%%%%%%%%%%%%%%%%%%

\bibliographystyle{amsplain}

\end{document}